\documentclass[11pt,oneside,reqno]{amsart}

\usepackage{enumerate}
\numberwithin{equation}{section}
\theoremstyle{definition}
\newtheorem{Definition}{Definition}[section]

\newtheorem{Remark}[Definition]{Remark}

\theoremstyle{plain}
\newtheorem{Theorem}[Definition]{Theorem}

\newtheorem{Proposition}[Definition]{Proposition}
\newtheorem{Corollary}[Definition]{Corollary}
\newtheorem{Lemma}[Definition]{Lemma}

\usepackage{amssymb}        % This automatically loads amsfonts.
\usepackage{mathrsfs}       % \mathscr will use the Ralph Smith’s For­mal Script.
\usepackage{eucal}          % \mathcal will use the Euler Script.
\usepackage{stmaryrd}       % St Mary’s Road symbol font. Contains \llbracket, \rrbracket, \mapsfrom.

\usepackage{geometry}
\usepackage{hyperref}

\usepackage{tikz}
\usetikzlibrary{arrows, matrix}
\usepackage{tikz-cd}
\usepackage{arydshln}       % Dash-lines in array/tabular

\newcommand{\ga}{\gamma}
\newcommand{\Ga}{\Gamma}

\newcommand{\ep}{\varepsilon}

\newcommand{\la}{\lambda}
\newcommand{\La}{\Lambda}

\renewcommand{\varphi}{\phi}
\newcommand{\om}{\omega}

\newcommand{\N}{\mathbb{N}}
\newcommand{\Z}{\mathbb{Z}}

\newcommand{\C}{\mathbb{C}}

\newcommand{\Fgl}{\mathfrak{gl}}
\newcommand{\Fsl}{\mathfrak{sl}}

\newcommand{\Fso}{\mathfrak{so}}

\newcommand{\Fm}{\mathfrak{m}}

\newcommand{\CL}{\mathcal{L}}
\newcommand{\CM}{\mathcal{M}}

\newcommand{\op}{\operatorname}

\DeclareMathOperator{\ann}{ann}

\DeclareMathOperator{\End}{End}

\DeclareMathOperator{\Frac}{Frac}

\DeclareMathOperator{\Hom}{Hom}

\DeclareMathOperator{\Span}{Span}

\DeclareMathOperator{\Specm}{Specm}

\DeclareMathOperator{\SL}{SL}

\renewcommand{\hat}{\widehat}
\renewcommand{\tilde}{\widetilde}

\title[Galois order realization of noncommutative type $D$ Kleinian singularities]{Galois order realization of noncommutative type $D$ Kleinian singularities}
\author{Jonas T. Hartwig}
\date{\today}

\address{Department of Mathematics, Iowa State University, Ames, IA-50011, USA}
\email{jth@iastate.edu}
\urladdr{http://jthartwig.net}

\begin{document}
\begin{abstract}
Galois orders, introduced by Futorny and Ovsienko, is a class of noncommutative algebras that includes generalized Weyl algebras, the enveloping algebra of the general linear Lie algebra and many others.
We prove that the noncommutative Kleinian singularities of type $D$ can be realized as principal Galois orders. Our starting point is an embedding theorem due to Boddington.
We also compute explicit generators for the corresponding (Morita equivalent) flag order, as a subalgebra of the nil-Hecke algebra of type $A_1^{(1)}$. 
Lastly, we compute structure constants for Harish-Chandra modules of local distributions and give a visual description of their structure from which subquotients are easily obtained.
\end{abstract}

\maketitle

\section{Introduction}

Kleinian singularities are singular algebraic surfaces which are quotients of $\C^2$ by a finite subgroup of $SL(2,\C)$. Via McKay correspondence the latter come in types $ADE$.
Noncommutative analogues of (the algebra of functions on) Kleinian singularities were introduced by Hodges \cite{Hod1993} (type $A$) and Crawley-Boevey and Holland \cite{CraHol1998} (types $ADE$). They also classified the finite-dimensional simple modules.
Special cases include the algebras of invariants of the first Weyl algebra with respect to a finite subgroup of $SL(2,\C)$, studied by Futorny and Schwarz  \cite{FutSch2018a}.
In \cite{Tik2011} the type $D$ algebras were shown to be related to infinitesimal Hecke algebras of $\Fsl_2$. 

\emph{Galois orders}, introduced by Futorny and Ovsienko \cite{FutOvs2010,FutOvs2014} are certain subrings of invariant rings of skew monoid rings. All known Galois orders are actually of a type called \emph{principal Galois orders}, introduced in \cite{Har2020}. Examples of principal Galois orders include many algebras of interest in representation theory such as generalized Weyl algebras, (quantized) enveloping algebras and finite W-algebras associated with the geneal linear Lie algebra \cite{FutMolOvs2010,FutHar2014,Har2020} and quantum Coulomb branches \cite{Web2019}. Representation theory of Galois orders has been studied in \cite{FutOvs2010},\cite{FutOvs2014},\cite{Har2020},\cite{Web2019}.
The type $A$ noncommutative Kleinian singularities are examples of generalized Weyl algebras \cite{Bav1991,Ros1995} and are therefore examples of principal Galois orders with trivial group $G$ as shown in \cite{FutOvs2010}. Webster \cite{Web2019} showed that any principal Galois order $U$ is the ``spherical subalgebra'' $eFe$ of what he termed a \emph{principal flag order} $F$. The flag orders are easier to understand, and (in most cases of interest) are Morita equivalent to the original algebra $U$.

In this paper we focus on noncommutative Kleinian singularities of type $D$, which are denoted by $D(q)$ where $q=q(t)$ is a polynomial parameter.
Presentations and isomorphism problems for $D(q)$ were studied in \cite{Bod2006} and \cite{Lev2009}. In particular, $D(q)$ has a presentation with three generators denoted $u,v,w$ and four relations, see \eqref{eq:DqRels}.
Using an embedding result due to Boddington \cite{Bod2006}, we realize $D(q)$ as a principal Galois $\C[u]$-order with symmetry group $G$ of order $2$. The subalgebra $\C[u]$ is a polynomial algebra in one variable, and is maximal commutative in $D(q)$, and $D(q)$ is free (of infinite rank) as a left and right $\C[u]$-module. Besides serving as interesting new examples of Galois orders, as a simple consequence we obtain a new proof of the Gelfand-Kirillov conjecture for these algebras (proved by different methods in \cite{Cra1999}), generalizing \cite{FutSch2018a}. Secondly, we study the corresponding principal flag order, that we denote $F(q)$. It is Morita equivalent to $D(q)$, and is by definition a certain subalgebra of the nil-Hecke algebra of type $A_1^{(1)}$. We compute a very simple set of generators for the algebra $F(q)$ (see Theorem \ref{thm:main2} for details).

We also prove a dichotomy theorem for a class of algebras including $D(q)$ and $F(q)$, stating that any simple module is either $\C[u]$-torsionfree or a Harish-Chandra module (meaning locally finite) with respect to $\C[u]$. This is a generalization of the dichotomy theorem for simple modules over generalized Weyl algebras over Dedekind domains \cite{BavVan1997}.

Lastly, we construct singular Harish-Chandra modules realized on spaces of formal distributions. This gives a breakdown of the simple modules into different classes, see Figure \ref{fig:classes}.

\begin{figure}
\centering
\begin{tikzpicture}
\tikzstyle{level 1}=[sibling distance=12em]
\tikzstyle{level 2}=[sibling distance=6em]
\tikzstyle{every node}=[rectangle,draw]
\node {simple $D(q)$-module}
   child { node {$\C[u]$-torsionfree} 
   }
   child { node {Harish-Chandra}
      child { node {generic} }
      child { node {singular}
         child { node {half-integral} }
         child { node {integral} } 
      }
   };
\end{tikzpicture}
\caption{Classes of simple $D(q)$-modules.}
\label{fig:classes}
\end{figure}
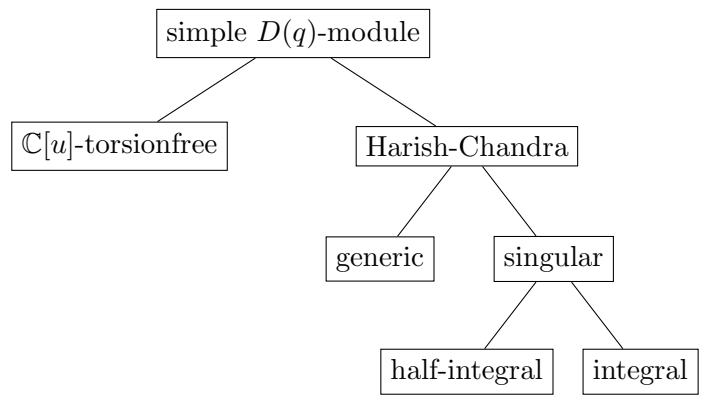

It is our hope that some of the techniques and features in this note could be of interest in future developments. An obvious question is whether the type $E$ noncommutative Kleinian singularities can also be realized as principal Galois orders, and furthermore, whether there is a conceptual and uniform realization of all ADE cases.

\subsection{Main Results}

\begin{Theorem} \label{thm:main1}
Let $q(t)$ be a polynomial of degree $n\ge 4$ and let $D(q)$ be the corresponding noncommutative Kleinian singularity. Then $D(q)$ can be realized as a principal Galois order. More precisely, letting $u,v,w$ denote the generators of $D(q)$ satisfying defining relations \eqref{eq:DqRels}, the following statements hold:
\begin{enumerate}[{\rm (i)}]
\item There exists an injective $\C$-algebra homomorphism
$\phi:D(q)\to \big(\C(x)\#\Z\big)^{S_2}$ given by
\begin{align}
\phi(u)&=x^2\\
\label{eq:phi-v-w}
\phi(-v-w)&=\frac{1}{2}\Big(\frac{q(x)}{\tfrac{1}{2}+x}\delta^{-1}+\frac{q(-x)}{\tfrac{1}{2}-x}\delta-\frac{q(-\tfrac{1}{2})}{(\tfrac{1}{2}+x)(\tfrac{1}{2}-x)}\Big) \\
\label{eq:phi-half-v-w}
\phi(-\tfrac{1}{2}v-w)&=\frac{1}{2}\Big(\frac{q(x)}{x}\delta^{-1}+\frac{q(-x)}{-x}\delta\Big)
\end{align}
where $\delta$ denotes a multiplicative generator of $\Z$, $\delta x = (x-1)\delta$, and the nontrivial element of $S_2$ negates $x$ and inverts $\delta$.
\item $\C(x^2)\phi\big(D(q)\big)=\big(\C(x)\#\Z\big)^{S_2}=\phi\big(D(q)\big)\C(x^2)$.
\item The image of $D(q)$ is contained in the subalgebra of $\big(\C(x)\#\Z\big)^{S_2}$ consisting of all elements which preserve $\C[x^2]$ with respect to the natural action of $\C(x)\#\Z$ on $\C(x)$.
\end{enumerate}
\end{Theorem}

Parts (i) and (ii) of Theorem \ref{thm:main1} state that $D(q)$ is a \emph{Galois ring} with respect to the subalgebra $\C[u]$, in the sense of Futorny and Ovsienko \cite{FutOvs2010}. Part (iii) of Theorem \ref{thm:main1} states that $D(q)$ is a principal Galois order in the sense of \cite{Har2020}.

The second result provides a simple description of an algebra Morita equivalent to $D(q)$.

\begin{Theorem}\label{thm:main2}
Let $\hat W\cong\Z\rtimes S_2=\langle \delta,\tau\mid \tau\delta\tau^{-1}=\delta^{-1},\,\tau^2=1\rangle$ be the affine Weyl group of type $A_1^{(1)}$, and let $\hat W$ act on $\C(x)$ by $\C$-algebra automorphisms determined by $\delta(x)=x-1$, $\tau(x)=-x$. Then
$D(q)$ is Morita equivalent to the subalgebra $F(q)$ of $\C(x)\#\hat W$ generated by
\begin{equation}
\big\{q(x)s_0,\; s_1,\; x\big\},
\end{equation}
where $s_i$ are the divided difference operators
\begin{equation}
s_1=\frac{1}{2x}(1-\tau),\qquad s_0=\frac{1}{2x+1}(1-\delta^{-1}\tau).
\end{equation}
\end{Theorem}

The third result says simple $D(q)$-modules are divided into two disjoint classes.

\begin{Theorem}
Let $\Ga$ be a commutative Harish-Chandra subalgebra of an algebra $U$. Suppose $\Ga$ is a noetherian integral domain of Krull dimension one. Then any simple $U$-module is either $\Ga$-torsionfree or a Harish-Chandra module with respect to $\Ga$.
\end{Theorem}

As a direct corollary we obtain:

\begin{Corollary}
Every simple $D(q)$-module is either $\C[u]$-torsionfree or a generalized weight module with respect to $\C[u]$. The same is true for simple $F(q)$-modules with respect to its subalgebra $\C[x]$.
\end{Corollary}

In our final section we describe bases and structure constants for certain infinite-dimensional representations of $D(q)$ on spaces of algebraic distributions. We refer to Section \ref{sec:reps} for details.

\subsection*{Acknowledgements}
The author was partially supported by Simons Foundation Collaboration grant \#637600 and Army Research Office  grant W911NF-24-1-0058.

\section{Type \texorpdfstring{$D$}{D} Noncommutative Kleinian Singularities}

\subsection{Definition}

Following Boddington \cite[Thm.~4.3]{Bod2006}, we make the following definition.

\begin{Definition}
Let $q(t)\in\C[t]$ be a polynomial of degree $n\ge 4$ in an indeterminate $t$.
The \emph{noncommutative type $D$ Kleinian singularity} associated to $q(t)$, denoted $D(q)$, is the associative unital $\C$-algebra with generators $u,v,w$ subject to defining relations
\begin{subequations}\label{eq:DqRels}
\begin{align}
\label{eq:DqRel1}
[u,v]&=2w+v,\\
\label{eq:DqRel2}
[u,w]&=2vu+w+\rho,\\
\label{eq:DqRel3}
[v,w]&=-v^2-p_1(u),\\
\label{eq:DqRel4}
w^2&=v^2u+vw+\rho v + p_0(u),
\end{align}
\end{subequations}
where
\begin{equation}
\rho=2q(-\tfrac{1}{2}),
\qquad
p(t) = \frac{-4q(t)q(-t-1)+\rho^2}{(1+2t)^2}\in\C[t],
\end{equation}
and $p_i(t)\in\C[t]$ are determined by $p(t)=p_0(t^2)+p_1(t^2)t$.
\end{Definition}

\begin{Remark}
The given definition of $D(q)$ is equivalent to the noncommutative type $D$ Kleinian singularities of Crawley-Boevey and Holland \cite{CraHol1998}. A proof of this fact can be found in \cite[Sec.~6]{Bod2006}.
\end{Remark}

\begin{Remark}
We have $\deg p(t)=2n-2$, $\deg p_1(t)=n-2$. Declaring $\deg(u)=4$ and $\deg(v)=2n-4$, $\deg(w)=2n-2$, the left hand sides of \eqref{eq:DqRel1}--\eqref{eq:DqRel3} have top degrees $2n, 2n+2, 4n-6$ respectively while the right hand sides have degrees $2n-2, 2n, 4n-8$ respectively. Thus, this equips $D(q)$ with a filtration whose associated graded algebra is commutative. This commutative algebra is isomorphic to the algebra of invariants in $\C[x,y]$ with respect to the type $D_n$ Kleinian subgroup of $\SL(2,\C)$, see \cite{Bod2006} for details.
\end{Remark}

\subsection{Properties of \texorpdfstring{$D(q)$}{D(q)}}

\begin{Definition}[{\cite{DroFutOvs1994}}]
A commutative subalgebra $\Ga$ of an algebra $U$ is a \emph{Harish-Chandra subalgebra} if $\Ga u\Ga$ is finitely generated as a left and right $\Ga$-module for each $u\in U$.
\end{Definition}

\begin{Proposition}\label{prp:Dq-properties}
The following statements hold.
\begin{enumerate}[{\rm (i)}]
\item $\{u^iv^jw^k\mid i,j,k\in\N,\, k\le 1\}$ and $\{v^jw^ku^i\mid i,j,k\in\N, k\le 1\}$ are $\C$-bases for $D(q)$.
\item $u$ is algebraically independent over $\C$. Hence $\C[u]$ is a polynomial ring in $u$ over $\C$.
\item $D(q)$ is free as a left and right $\C[u]$-module.
\item If $X\in D(q)$ is such that $X\cdot f\in \C[u]$ for some nonzero $f\in\C[u]$, then $X\in\C[u]$.
\item There is an anti-isomorphism $\ast:D(q)\to D(q)$ given by $u^\ast=u, v^\ast=v, w^\ast=-w-v$.
\item $\C[u]$ is a Harish-Chandra subalgebra of $D(q)$.
\end{enumerate}
\end{Proposition}

\begin{proof}
(i) That the first set is a basis follows from Levy \cite[p.6]{Lev2009} (after a simple affine change of generators given in \cite[p.9]{Bod2006}). That the second set is a basis is proved in \cite[p.7]{Bod2006}.

(ii) and (iii) are immediate by (i).

(iv) See \cite[p.7]{Bod2006}.

(v) Straightforward to check from the defining relations \eqref{eq:DqRels}. One can also use Boddington's embedding to show this, see Remark \ref{rem:beta-theta}.

(vi) Let $B=\C[u]$. Since $B$ is noetherian and
$B(x+y)B\subseteq BxB+ByB$ and $BxyB\subseteq (BxB)(ByB)$ for all $x,y\in D(q)$, it suffices to show that $BvB$ and $BwB$ are contained in finitely generated as left and right $B$-submodules of $D(q)$. By \eqref{eq:DqRel1}--\eqref{eq:DqRel2} we have
\begin{gather}
uv = v(u+1)+w\cdot 2, \\ 
uw = w(u+1)+v\cdot(2u)+\rho.
\end{gather}
By induction on $k$ we have $u^k v\in vB+wB+1B$ and $u^k w \in vB + wB+1B$. Thus $BvB+BwB \subset vB + wB+1B$. Similarly one shows that $BvB+BwB\subset Bv+Bw+B1$.
\end{proof}

\subsection{Boddington's Embedding}
\label{sec:Boddington}

For $s(t)\in\C(t)$ let $\tilde{T}(s)$ be ring extension of the algebra $\C(h)$ of rational functions in an indeterminate $h$, by two elements $a,b$ subject to relations:
\begin{subequations}\label{eq:TsRels}
\begin{align}
ba&=s(h),\\
ab&=s(h-1),\\
ah&=(h-1)a, \label{eq:TsRel3}\\
bh&=(h+1)b. \label{eq:TsRel4}
\end{align}
\end{subequations}
Thus $\tilde{T}(s)$ is just the generalized Weyl algebra $\C(h)(\sigma,s)$ where $\sigma(h)=h-1$, see \cite{Bav1991}.

\begin{Theorem}[Boddington \cite{Bod2006}]
\label{thm:beta}
Given a polynomial $q(t)\in\C[t]$ of degree $n\ge 4$, let
\begin{equation} \label{eq:s-def}
s(t)=\frac{q(t)q(-t-1)}{t(t+1)(1+2t)^2} \in \C(h).
\end{equation}
There exists an injective algebra homomorphism
\begin{subequations} \label{eq:Boddington-map}
\begin{equation} 
\beta: D(q)\to \tilde{T}(s),
\end{equation}
given by:
\begin{align}
\beta(u) &= h^2, \\
\beta(v) &= a+b+\frac{2\rho}{1-4h^2}, \label{eq:beta3}\\
\beta(w) &= (a-b)h-\frac{\rho}{1-4h^2}. \label{eq:beta4}
\end{align}
\end{subequations}
\end{Theorem}

\begin{Remark}\label{rem:beta-theta}
It is easy to see that the algebra $\tilde{T}(s)$ (being a generalized Weyl algebra) has an anti-automorphism $\ast$ given by $h^\ast=h,\; a^\ast=b,\; b^\ast=a$. Moreover, the image of $\beta$ is preserved by this anti-automorphism: $\beta(u)$ and $\beta(v)$ are fixed by $\ast$ while $\beta(w)^\ast=-\beta(w)-\beta(v)$ using \eqref{eq:beta3},\eqref{eq:beta4}. Since $\beta$ is injective, $x\mapsto \beta^{-1}(\beta(x)^\ast)$ defines an anti-automorphism of $D(q)$, which coincides with $\ast$ from Proposition \ref{prp:Dq-properties}(v).
\end{Remark}

\section{Principal Galois Orders}

\subsection{Definition}
Galois orders were defined in \cite{FutOvs2010}. We recall the definition of principal Galois (respectively flag) orders from \cite{Har2020} (respectively \cite{Web2019}).
Let $\La$ be a noetherian integrally closed domain with field of fractions $L$. Let $\mathcal{M}$ be a monoid acting faithfully by ring automorphisms on $\La$ (hence on $L$).
Let 
\begin{equation}
\mathcal{L}=L\#\mathcal{M}
\end{equation}
be the skew monoid ring. It is the free left $L$-module on the set $\mathcal{M}$ with multiplication determined by $(a\mu)\cdot (b\nu)=a\mu(b)\mu\nu$ ($a,b\in L,\,\mu,\nu\in\mathcal{M}$). We have natural embeddings of rings:
\begin{equation}\label{eq:L-diagram1}
\La \longrightarrow L\longrightarrow \mathcal{L}.
\end{equation}
Let $W$ be a finite group, acting faithfully by ring automorphisms of $\mathcal{L}$ preserving $\La$ and $\CM$.\footnote{These assumptions force $W$ to act on $\CM$ by conjugation and $\CM\CM^{-1}\cap W=1$.}
The diagram \eqref{eq:L-diagram1} fits into a larger commutative diagram of injective ring homomorphisms:
\begin{equation}\label{eq:L-diagram2}
\begin{tikzcd}
\La\# W \arrow{r}    & L\# W \arrow{r}    & \CL\# W \\
\La     \arrow{u}\arrow{r} & L     \arrow{u}\arrow{r} & \CL \arrow{u}\\
\La^W   \arrow{u}\arrow{r} & L^W   \arrow{u}\arrow{r} & \CL^W \arrow{u}
\end{tikzcd}
\end{equation}
For a minor rephrasing of the definition of principal flag order, we need an observation.
\begin{Lemma}
\begin{enumerate}[{\rm (a)}]
\item The unique left $L\# W$-module homomorphism $\ep:\CL\# W\to L\# W$ determined by $\ep(\mu)=1$ for all $\mu\in\CM$ satisfies
\begin{equation}
\ep(X\ep(Y)) = \ep(XY),\qquad\forall X,\,Y\in\CL\#W.
\end{equation}
\item There is a ring homomorphism $\CL\# W\to\End(L\# W)$, $X\mapsto \hat X$, given by
\begin{equation}
\hat X(a) = \ep(Xa), \qquad \forall X\in\CL\# W,\, a\in L\# W.
\end{equation}
\end{enumerate}
\end{Lemma}

\begin{proof}
(a) By $\Z$-bilinearity, it suffices to verify this for $X=fw\mu, Y=f'w'\mu'$ in which case both sides simplify to $f\cdot w\cdot\mu(f')\cdot w'$. 

(b) $(\hat X\circ\hat Y)(a)=\ep\big(X\ep(Ya))$ while $\hat{XY}(a)=\ep(XYa)$. Use (a) with $Ya$ for $Y$.
\end{proof}

Using this action of $\CL\# W$ on $L\# W$ we can state the definitions of principal flag/Galois order as follows.

\begin{Definition}\phantom{X}
\begin{enumerate}[{\rm (i)}]
\item A $\La\# W$-subring $F$ of $\CL\# W$ is a \emph{principal flag order} if every $X\in F$ preserves $\La\# W$ and $(L\# W)F=\CL\# W$.
\item A $\La^W$-subring $U$ of $\CL^W$ is a \emph{principal Galois order} if every $X\in U$ preserves $\La^W$ and $L^W U=\CL^W$.
\end{enumerate}
\end{Definition}

\begin{Remark}
A $\La\# W$-subring $F$ of $\CL\# W$ is a principal flag order if and only if every $X\in F$ preserves $\La$ and $LF=\CL\# W$.
\end{Remark}

\subsection{Realization of \texorpdfstring{$D(q)$}{D(q)} as a Principal Galois Order}

Let $\La=\C[x]$ be the algebra of polynomials in one variable $x$ and $L=\C(x)$ its field of fractions.
Let $W=\langle \tau\rangle\cong S_2$ be the group of order two with generator $\tau$, acting by $\C$-algebra automorphisms on $\La$ via $\tau(p(x))=p(-x)$.
Let $\mathcal{M}=\langle\delta\rangle\cong\Z$ act on $\La$ via shift automorphisms: $\delta(f(x))=f(x-1)$ for all $f(x)\in\C[x]$. Then $\tau\circ\delta^k\circ\tau^{-1}=\delta^{-k}$ for all $k\in\Z$, so $W$ normalizes $\mathcal{M}$.
Let $L\#\mathcal{M}=\C(x)\#\Z$ be the smash product (skew group algebra) of $\Z$ over $\C(x)$, $\La^W=\C[x^2]$, $L^W=\C(x^2)$, and $(L\#\CM)^W$ the corresponding subalgebras of $W$-invariants.

In the following lemma we observe that, for any $s(t)$ of the form \eqref{eq:s-def}, there is an isomorphism $\tilde{T}(s)\cong L\#\Z$ that behaves well with $W$.

\begin{Lemma} \label{lem:psi}
Let $q(t)\in\C[x]$ be a polynomial of degree $\ge 4$, and let $s(t)$ be as in \eqref{eq:s-def}. Let $\tilde T(s)$ be the corresponding algebra from Section \ref{sec:Boddington}.
There exists an isomorphism
\begin{subequations}\label{eq:psi-isomorphism}
\begin{equation}
\psi:\tilde{T}(s)\to \C(x)\#\Z,
\end{equation}
determined by
\begin{equation}
\psi(h)=x,\qquad \psi(a)=f(x)\delta, \qquad \psi(b)=f(-x)\delta^{-1},
\end{equation}
\end{subequations}
where
\begin{equation}
f(x)=\frac{1}{2}\frac{q(-x)}{(-x)(\tfrac{1}{2}-x)}\in\C(x).
\end{equation}
Moreover, we have
\begin{equation}\label{eq:tau-on-hab}
\tau(x)=-x,\qquad \tau(\psi(a))=\psi(b),\qquad \tau(\psi(b))=\psi(a).
\end{equation}
\end{Lemma}
\begin{proof} It suffices to check that the relations \eqref{eq:TsRels} are preserved. We have
\begin{equation}
\psi(b)\psi(a) = f(-x)\delta^{-1}f(x)\delta=f(-x)f(x+1)=s(x)=\psi\big(s(h)\big).
\end{equation}
The others are shown similarly. Lastly, the equalities \eqref{eq:tau-on-hab} are immediate by the definition of $\tau$.
\end{proof}

Combining Boddington's embedding with the isomorphism $\psi$ just constructed, we obtain an embedding of $D(q)$ in the smash product $\C(x)\# \Z$ with nice properties.

\begin{Lemma} \label{lem:psibeta}
Let $\beta$ be Boddington's embedding \eqref{eq:Boddington-map} and $\psi$ be the isomorphism \eqref{eq:psi-isomorphism}. Let $U$ be the image in $\C(x)\#\Z$ of the composition $\psi\circ\beta$. Then
\begin{enumerate}[{\rm (i)}]
\item $U$ is contained in the invariant subalgebra $\big(\C(x)\#\Z\big)^{S_2}$,
\item $a^n+b^n$ and $(a^n-b^n)x$ belong to $\C(x^2)U$ for all integers $n>0$,
\item $\C(x^2)U=\big(\C(x)\#\Z\big)^{S_2}=U\C(x^2)$.
\end{enumerate}
\end{Lemma}

\begin{Remark}
In the terminology of \cite[Definition 3]{FutOvs2010}, (i) and (iii) say that $U$ is a \emph{Galois ring} with respect to $\C[x^2]$.
\end{Remark}

\begin{proof}
(i) This follows from the fact that $\tau$ interchanges $a$ and $b$ while sending $h$ to $-h$.

(ii) Let $U'=\C(x^2)U$. By definition of $\varphi$, $U'$ contains the elements $a+b$ and $(a-b)h$.
Hence $U'$ also contains $(a+b)^2 = a^2+ab+ba+b^2$. Since $ab$ and $ba$ are both rational functions fixed by $\tau$ it is easy to see that they belong to $\C(h^2)$, which shows that $U'$ contains $a^2+b^2$. Similarly one shows that $a^n+b^n$ and $(a^n-b^n)h$ belong to $U'$ for all integers $n>0$.

(iii)
Put
 \[x_n=\begin{cases}a^n,&n>0\\ 1,&n=0\\b^{|n|},&n<0\end{cases}\]
and let $\sum_n x_nf_n(h)$ be an arbitrary element in the invariant subring $\mathcal{K}$. Since $\tau(h)=-h$, and $\tau(x_n)=x_{-n}$, clearly
\begin{equation}\label{eq:fn1}
f_n(-h)=f_{-n}(h).
\end{equation}
Now write
\begin{equation}
f_n(h) = g_n^+(h^2) + g_n^-(h^2)h.
\end{equation}
By \eqref{eq:fn1} we have get
\begin{equation}
g_n^+(h^2) - g_n^-(h^2)h = g_{-n}^+(h^2) + g_{-n}^-(h^2)h 
\end{equation}
hence
\begin{equation}
g_n^\pm(h^2)=\pm g_{-n}^\pm(h^2)
\end{equation}
which implies that $g_0^-(h^2)=0$ and
\begin{equation}
\sum_n x_n f_n(h) = g_0^+(h^2) + \frac{1}{2}\sum_{n>0}  (x_n+x_{-n})g_n^+(h^2)  + \frac{1}{2}\sum_{n>0}  (x_n-x_{-n}) \cdot h g_n^-(h^2)
\end{equation}
Recalling that $x_n+x_{-n} = a^n+b^n$, the claim follows by the previous part.
\end{proof}

We are now ready to prove the first main result.

\begin{proof}[{Proof of Theorem \ref{thm:main1}.}]
Define $\varphi=\psi\circ\beta$. Using Theorem \ref{thm:beta}, Lemma \ref{lem:psi} we have
\begin{align}
\varphi(u)&=x^2,\\
\varphi(v) &= \frac{1}{2}\Big(\frac{q(x)}{x(\tfrac{1}{2}+x)}\delta^{-1}+\frac{q(-x)}{(-x)(\tfrac{1}{2}-x)}\delta\Big)+\frac{q(-\tfrac{1}{2})}{(\tfrac{1}{2}+x)(\tfrac{1}{2}-x)},\\ 
\varphi(w) &=\frac{1}{2}\Big(\frac{q(x)(-1-x)}{x(\tfrac{1}{2}+x)}\delta^{-1}+\frac{q(-x)(x+1)}{(-x)(\tfrac{1}{2}-x)}\delta\Big) +\frac{q(-\tfrac{1}{2})(-\tfrac{1}{2})}{(\tfrac{1}{2}+x)(\tfrac{1}{2}-x)}.
\end{align}
By Lemma \ref{lem:psibeta}(a) the image is contained in $\big(\C(x)\#\Z\big)^{S_2}$. Using partial fraction decomposition, direct computations prove (i). Lemma \ref{lem:psibeta}(c) implies (ii). Lastly, for any even polynomial $p(x)$, letting $.$ stand for the action of $L\#\Z$ on $L$,
\[\big(q(x)\delta^{-1}-q(-x)\delta\big).p(x)=q(x)p(x+1)-q(-x)p(x-1)\]
is an odd polynomial, hence is divisible by $2x$. Similarly,
\[\big(q(\pm x)\delta^{\mp 1}-q(-\tfrac{1}{2})\big).p(x)=q(\pm x)p(x\pm 1)-q(-\tfrac{1}{2})p(x)\]
has a zero at $x=\mp \tfrac{1}{2}$ hence is divisible by $1\pm 2x$. This proves (iii).
\end{proof}

By \cite{Har2020}, an immediate corollary is:
\begin{Corollary} \label{lem:maxcomm}
$\C[u]$ is maximal commutative in $D(q)$.
\end{Corollary}

\subsection{Gelfand-Kirillov Conjecture}

Let $\Ga$ be a finite subgroup of $\SL_2(\C)$ of type $D$. Let $A_1(\C)$ be the first Weyl algebra over $\C$. Then $\Ga$ acts naturally by algebra automorphisms of $A_1(\C)$. The corresponding subalgebra $(A_1)^\Ga$ of $\Ga$-invariants is a special case of a noncommutative type $D$ Kleinian singularity \cite{CraHol1998}. Thus the following corollary generalizes \cite{FutSch2018a} in the type $D$ case.

\begin{Corollary}
$D(q)$ satisfies the Gelfand-Kirillov conjecture. Explicitly we have
\begin{equation}
\Frac D(q) \cong \Frac A_1(\C).
\end{equation}
\end{Corollary}
\begin{proof} The technique here is well-known. Following for example \cite{FutOvs2014}, we have
\begin{align*}
\Frac D(q) &\cong \Frac\big( (\C(x)\# \Z)^{S_2}\big) \\
&\cong \big(\Frac (\C(x)\# \Z)\big)^{S_2} \\
&\cong \big(\Frac A_1(\C)\big)^{S_2} \\
&\cong \Frac\big( A_1(\C)^{S_2}\big).
\end{align*}
Since $\C(x)^{S_2}=\C(x^2)\cong \C(x)$, $S_2$ solves the commutative Noether problem, hence, by for example \cite{FutSch2019},
$\Frac\big( A_1(\C)^{S_2}\big) \cong \Frac A_1(\C)$.
\end{proof}

\begin{Remark}
This result was proved in all types ADE by different methods in \cite{Cra1999}.
\end{Remark}

\section{The Principal Flag Order Associated to \texorpdfstring{$D(q)$}{D(q)}}

We use the obtained Galois order realization of $D(q)$ to prove it is Morita equivalent to an explicit subalgebra of the nil-Hecke algebra of type $A_1^{(1)}$.

As before, consider the skew group algebra (smash product) $\C(x)\#\Z\rtimes S_2$ of the field $\C(x)$ of rational functions with the group $\Z\rtimes S_2$ (Weyl group of type $A_1^{(1)}$). Writing $\Z$ multiplicatively with generator $\delta$, and $S_2=\{(1),\tau\}$, the relations are
\begin{equation}
\tau\cdot f(x)=f(-x)\cdot \tau,\qquad \delta\cdot f(x)=f(x-1)\cdot\delta,\qquad \tau\delta\tau^{-1}=\delta^{-1},\qquad \tau^2=1.
\end{equation}
The nil-Hecke algebra $N$ of type $A_1^{(1)}$ can be defined as the subalgebra of $\C(x)\#\Z\rtimes S_2$ generated by the three elements
\[x,\quad s_1=\frac{1}{2x}(1-\tau),\quad s_0=\frac{1}{2x+1}(1-\delta^{-1}\tau).\]
The natural action of $L\#\Z\rtimes S_2$ on $L$ restricts to a faithful action of $N$ on $\C[x]$ by divided difference operators:
\begin{equation}
s_0(p(x))=\frac{p(x)-p(-1-x)}{x-(-1-x)},\qquad
s_1(p(x))=\frac{p(x)-p(-x)}{x-(-x)}.
\end{equation}
In fact, the image of the homomorphism $N\to\End_\C(\C[x])$ equals $\End_{\C[x^2]}(\C[x])$, see \cite{Kum2012}.
The two elements $s_0$ and $s_1$ satisfy
\begin{equation}
s_0^2=0,\quad s_1^2=0,\quad s_0s_1s_0s_1=s_1s_0s_1s_0.
\end{equation}
Let $e=\frac{1}{2}(1+\tau)$ be the symmetrizing idempotent for $S_2$. Fix a polynomial $q(t)$, $\deg q(t)\ge 4$.
Let $U=U(q)=\phi\big(D(q)\big)$, where $\phi$ is the injective algebra map from Theorem \ref{thm:main1}. Let $F(q)$ be the subalgebra of $\C(x)\#\Z\rtimes S_2$ generated by the subset
\begin{equation}
eU(q)e \,\cup\, \{s_1,\, x\}.
\end{equation}
By \cite{Web2019}, $F(q)$ is a principal flag order and is furthermore Morita equivalent to $U(q)$ (and hence to $D(q)$). The most important identity in this context, which establishes that $\End_{\C[x^2]}(\C[x])$ is Morita equivalent to $\C[x^2]$ reads
\begin{equation}\label{eq:nilHecke}
e+xe\frac{1}{x} = 1.
\end{equation}
Other useful relations include
\begin{equation}
s_1x+xs_1=1,\qquad s_1x-xs_1=\tau,\qquad s_1x=e, \qquad s_1=es_1.
\end{equation}

\begin{Lemma} \label{lem:F-generators}
Put $a=\phi(-\frac{1}{2}v-w)$ and $b=\phi(-v-w)$. The following statements hold:
\begin{enumerate}[{\rm (a)}] 
\item $F(q)=\C\langle eae,\, ebe,\, s_1, x\rangle$,
\item $q(x)\delta^{-1}\in F(q)$ and moreover $F(q)=\C\langle q(x)\delta^{-1},\, ebe,\, s_1,\, x\rangle$,
\item $D:=\frac{1}{\frac{1}{2}+x}\big(q(x)\delta^{-1}-q(-1/2)\tau\big)\in F(q)$ and moreover $F(q)=\C\langle q(x)\delta^{-1},\, D,\, s_1,\, x\rangle$,
\item $[D,x^2]=2q(x)\delta^{-1}$, so $F(q)=\C\langle D,\, s_1,\, x\rangle$,
\item $q(x)s_0 = -\frac{1}{2}D\tau-\frac{q(-1/2)-q(x)}{2x+1}$ and therefore
\begin{equation}
F(q)=\C\langle q(x)s_0,\, s_1,\, x\rangle.
\end{equation}
\end{enumerate}
\end{Lemma}

\begin{proof}
(a) This is immediate, since $U$ is generated by$\{a,b,x^2\}$.

(b) We have
\begin{align*}
q(x)\delta^{-1} &=(e+xe\frac{1}{x})q(x)\delta^{-1}(e+xe\frac{1}{x})\\
&=xe\frac{1}{2}\big(q(x)\delta^{-1}+q(-x)\delta\big)e\\
&+xe\frac{1}{2}\big(\frac{q(x)}{x}\delta^{-1}+\frac{q(-x)}{-x}\delta\big)e\\
&+e\frac{1}{2}\big(q(x)(x+1)\delta^{-1}+q(-x)(-x+1)\delta\big)e\frac{1}{x}\\
&+xe\frac{1}{2}\big(\frac{q(x)(x+1)}{x}\delta^{-1}+\frac{q(-x)(-x+1)}{-x}\delta\big)e\frac{1}{x}
\end{align*} 
We show that each of the four terms belong to $F(q)$.
The second term equals $xeae$ by \eqref{eq:phi-half-v-w}. A computation shows that $[a,x^2]=q(x)\delta^{-1}+q(-x)\delta$, so the first term equals $xe[a,\phi(u)]e\in F(q)$. Applying $[\cdot,x^2]$ once more one checks the third term is in $F(q)$, also using $e\frac{1}{x}=s=es$. The fourth term can be expressed in terms we already know belong to $F(q)$, using $\frac{x+1}{x}=1+\frac{1}{x}$. This proves that $q(x)\delta^{-1}\in F(q)$, and hence the $\supseteq$ inclusion. For the reverse inclusion, one checks that $eae=s_1q(x)\delta^{-1}e$.

(c) First note that since $x=\frac{1}{2}+x-\frac{1}{2}$, we have
\begin{equation}\label{eq:generator-lemma-1}
\frac{x}{\frac{1}{2}+x}\big(q(x)\delta^{-1}-q(-1/2)\tau\big) = q(x)\delta^{-1}-q(-1/2)\tau + \frac{-1/2}{\frac{1}{2}+x}\big(q(x)\delta^{-1}-q(-1/2)\tau\big)
\end{equation}
By (b) and the fact that $\tau\in F(q)$, to prove (c) it suffices to show the left hand side of 
\eqref{eq:generator-lemma-1} belongs to $F(q)$. We have by \eqref{eq:nilHecke}:
\begin{align*}
\frac{x}{\frac{1}{2}+x}\big(q(x)\delta^{-1}-q(-1/2)\tau\big) &= (e+xe\frac{1}{x}) 
\frac{x}{\frac{1}{2}+x}\big(q(x)\delta^{-1}-q(-1/2)\tau\big) (e+xe\frac{1}{x}) \\
&= e\frac{x}{\frac{1}{2}+x}\big(q(x)\delta^{-1}-q(-1/2)\tau\big) e \\
&+ xe\frac{1}{\frac{1}{2}+x}\big(q(x)\delta^{-1}-q(-1/2)\tau\big) e \\
&+e\frac{x}{\frac{1}{2}+x}\big(q(x)(x+1)\delta^{-1}+xq(-1/2)\tau\big) e\frac{1}{x}\\
&+xe\frac{1}{\frac{1}{2}+x}\big(q(x)(x+1)\delta^{-1}+xq(-1/2)\tau\big) e\frac{1}{x}.
\end{align*}
We prove that each of these four term belongs to $F(q)$.
The second term belongs to $F(q)$ since it is equal to $x\cdot e \phi(-v-w)e\in F(q)$ (the $\tau$ can be absorbed into $e$, and the expression between the $e$'s can be symmetrized). By the division algorithm $\frac{x}{\frac{1}{2}+x} = 1-\frac{1/2}{\frac{1}{2}+x}$, the first term belongs to $F(q)$ as well, using part (b) and that $e,\tau\in F(q)$. To prove the fourth term is in $F(q)$, we use
\[\frac{1}{\frac{1}{2}+x}\big(q(x)(x+1)\delta^{-1}+xq(-1/2)\tau\big)=2q(x)\delta^{-1}-\frac{x}{\frac{1}{2}+x}\big(q(x)\delta^{-1}-q(-1/2)\tau\big)\]
along with part (b), that $e\frac{1}{x}=s_1=es_1$, $s_1\in F(q)$, and that we already proved the first term belongs to $F(q)$. Lastly, the third term belongs to $F(q)$ using the division algorithm similarly to how we proved the first term is in $F(q)$. We have shown $D\in F(q)$, proving the inclusion $\supseteq$. For the reverse, we have $ebe=eDe$ by direct calculation.

(d) and (e) follow from direct computation.
\end{proof}

We can now prove the second main theorem from the introduction.

\begin{proof}[Proof of Theorem \ref{thm:main2}]
By Lemma \ref{lem:F-generators}, the principal flag order $F(q)$ is generated by $\{q(x)s_0,\,s_1,\,x\}$. As already stated, by general results on flag orders due to Webster \cite{Web2019}, $F(q)$ is Morita equivalent to $\phi(D(q))$, hence to $D(q)$ since $\phi$ is an injective algebra map.
\end{proof}

\section{Harish-Chandra Modules over \texorpdfstring{$D(q)$}{D(q)}}

\begin{Definition}[{\cite{DroFutOvs1994}}]
A commutative subalgebra $\Ga$ of an algebra $A$ is a \emph{Harish-Chandra subalgebra} if $\Ga a\Ga$ is finitely generated as a left and right $\Gamma$-module for any $a\in A$.
\end{Definition}

\begin{Definition}[{\cite{DroFutOvs1994}}]
Let $\Ga$ be a Harish-Chandra subalgebra of an algebra $U$. A finitely generated $U$-module $V$ is a \emph{Harish-Chandra module} if $\Ga v$ is finite-dimensional for every $v\in V$. When $\Ga$ is noetherian, this is equivalent to that $V$ has a decomposition
\[
V=\bigoplus_{\Fm\in\Specm(\Ga)} V_\Fm,\qquad V_\Fm=\{v\in V\mid \Fm^k v=0,\; k\gg 0\}.
\]
\end{Definition}

In particular, specializing to the case of $A=D(q)$ and $\Ga=\C[u]$:

\begin{Definition}
A finitely generated $D(q)$-module $V$ is a \emph{Harish-Chandra module} if
\[
V=\bigoplus_{\chi\in\C} V_\chi,\qquad V_\chi=\{a\in V\mid (u-\chi)^k a=0, k\gg 0\}.
\]
\end{Definition}

\subsection{Dichotomy Theorem for Simple \texorpdfstring{$D(q)$}{D(q)}-Modules}

A module $V$ over a commutative integral domain $\Ga$ is \emph{$\Ga$-torsionfree} if for for every nonzero $v\in V$ the map $\Ga\to V$, $\ga\mapsto \ga.v$ is injective.

The following theorem generalizes the Dichotomy Theorem for simple modules over generalized Weyl algebras over Dedekind domains \cite{BavVan1997}. The idea is a special case of general principles of prime height stratification, see for example \cite{FutOvsSao2011}.

\begin{Theorem}\label{thm:dichotomy}
Let $\Ga$ be a commutative Harish-Chandra subalgebra of an algebra $U$. Suppose $\Ga$ is a noetherian integral domain of Krull dimension one. Then any simple $U$-module is either $\Ga$-torsionfree or a Harish-Chandra module with respect to $\Ga$.
\end{Theorem}

\begin{proof}
Let $V$ be a simple $U$-module. Consider the restriction of $V$ to $\Gamma$. For $v\in V$, consider the annihilator $\ann_\Ga(v)=\{\ga\in\Ga\mid \ga v=0\}$.
Suppose that $V$ is not $\Ga$-torsionfree. That is, suppose there exists a nonzero vector $v\in V$ and a nonzero element $\ga\in \Ga$ such that $\ga v=0$.
That means that there are nonzero ideals in the family
\[\{\ann_\Ga(v)\mid v\in V\setminus\{0\}\}\]
of proper ideals of $\Ga$. The maximal elements of this set are (associated) prime ideals. 
This proves that if $V$ is not $\Ga$-torsionfree then there exists a nonzero prime ideal $\mathfrak{p}$ of $\Ga$ such that $\mathfrak{p}v=0$ for some nonzero $v\in V$. 
Since $\Ga$ is an integral domain in which every nonzero prime is maximal, this proves that the direct sum of generalized $\Ga$-weight spaces
\[V'=\bigoplus_{\Fm\in\op{Specm(\Ga)}} V_\Fm,\qquad V_\Fm=\{v\in V\mid \Fm^n v=0,\, n\gg 0\}.\]
is a nonzero subspace of $V$. Since $\Ga$ is a Harish-Chandra subalgebra of $U$, the subspace $V'$ is a $U$-submodule of $V$ by \cite[Prop.~14]{DroFutOvs1994}. Since $V$ is simple, $V=V'$.
This proves that $V$ is either $\Ga$-torsionfree, or a Harish-Chandra module with respect to $\Ga$.
\end{proof}

\begin{Corollary}
Every simple $D(q)$-module is either a $\C[u]$-torsionfree module, or a Harish-Chandra module with respect to $\C[u]$.
\end{Corollary}

\begin{proof}
Immediate by Theorem \ref{thm:dichotomy} and Proposition \ref{prp:Dq-properties}(ii),(vi).
\end{proof}{}

\subsection{Harish-Chandra Modules of Local Distributions}
\label{sec:reps}

The results in this section are analogous to computations for $U(\Fgl_n)$ in the setting of derivative tableaux from \cite{FutGraRam2016} and local distribution approach from \cite{Vis2018}.
Recall that $D(q)$ acts from the left on $\C[x^2]$ via $\varphi$ by Theorem \ref{thm:main1}. Thus we may use the anti-automorphism $\ast$ from Proposition \ref{prp:Dq-properties}(v) to equip the dual space $\Hom_\C(\C[x^2],\C)$ with the structure of a left $D(q)$-module. Explicitly, 
\begin{equation}
(X.\xi)\big(p(x)\big)=\xi\big(\varphi(X^\ast).p(x)\big),\qquad \forall X\in D(q),\; \xi\in\Hom_\C(\C[x^2],\C),\; p(x)\in\C[x^2].
\end{equation}

\begin{Definition}
For any $\la\in\C$ the \emph{tableaux} $T_0(\la)$ and \emph{derivative tableaux} $T_1(\la)$ are the elements of $\Hom_\C\big(\C[x^2],\C\big)$  defined by 
\begin{equation}
T_0(\la)\big(p(x)\big) = p(\la), \qquad T_1(\la)\big( p(x)\big) = p'(\la),\qquad\forall p(x)\in\C[x^2].
\end{equation}
\end{Definition}

\begin{Remark}
Note that for all $\la\in\C$,
\begin{equation}
T_0(-\la) = T_0(\la),\qquad T_1(-\la)=-T_1(\la).
\end{equation}
In particular, $T_1(0)=0$.
\end{Remark}

In analogy with Gelfand-Tsetlin patterns for $\Fgl_n$ or $\Fso_n$ where the top row determines the highest weight, hence a central quotient of the enveloping algebra, one could picture these tableaux as follows:
\begin{equation}
T_0(\la)=\,
\begin{tikzpicture}[scale=0.7,baseline=17pt,thick]
\draw (0,1)--(0,0)--(2,0)--(2,1);
\draw (1,1)--(1,0);
\draw (-.5,1)--(2.5,1)--(2.5,2)--(-.5,2)--cycle;
\node at (0.5,0.5) {$\lambda$};
\node at (1.5,0.5) {$-\lambda$};
\node at (1,1.5)   {$q(x)$};
\end{tikzpicture}
\qquad
T_1(\la)=\begin{tikzpicture}[scale=0.7,baseline=17pt,thick,dashed]
\draw (0,1)--(0,0)--(2,0)--(2,1);
\draw (1,1)--(1,0);
\draw (-.5,1)--(2.5,1)--(2.5,2)--(-.5,2)--cycle;
\node at (0.5,0.5) {$\lambda$};
\node at (1.5,0.5) {$-\lambda$};
\node at (1,1.5)   {$q(x)$};
\end{tikzpicture}
\end{equation}

For each coset $\omega\in\C/\Z$ consider the subspaces
\begin{equation}
M(\omega)=\Span_\C\big\{T_0(\la),T_1(\la)\mid \la\in\omega\big\}\subset\Hom_\C(\C[x^2],\C),
\end{equation} 
\begin{equation}
M_0(\omega)=\Span_\C\big\{T_0(\la)\mid \la\in\omega\big\}\subset\Hom_\C(\C[x^2],\C).
\end{equation}
The following result describes the action of the generators $\{u,\,\tfrac{1}{2}v+w,\,w\}$ of $D(q)$ on $M(\omega)$ in the tableaux basis. In particular it shows that $M(\omega)$ is a submodule of $\Hom_\C\big(\C[x^2],\C\big)$, and $M_0(\omega)$ is likewise a submodule, provided $\omega\neq \Z, \frac{1}{2}+\Z$. The formulas also provide a visual description of these modules, see Figures \ref{fig:generic} for the generic cases, Figure \ref{fig:integral} for the integral case, and Figure \ref{fig:half-integral} for the half-integral case.

\begin{Proposition}\label{prop:calculations}
\begin{subequations}
\begin{enumerate}[{\rm (a)}]
\item For any $\la\in\C$,
\begin{equation} \label{eq:u-action}
(u-\la^2).T_0(\la)=0,\qquad (u-\la^2).T_1(\la)=T_0(\la).
\end{equation}
\item If $\la\in\C\setminus\{0\}$, then
\begin{align}
(\tfrac{1}{2}v+w).T_0(\la)&=
\frac{q(\la)}{2\la}T_0(\la+1)-\frac{q(-\la)}{2\la}T_0(\la-1), \label{eq:v-action-T0} \\
(\tfrac{1}{2}v+w).T_1(\la)&=
\frac{q(\la)}{2\la}T_1(\la+1)+\frac{q(-\la)}{2\la}T_1(\la-1) \nonumber \\
&+\frac{q'(\la)-q(\la)}{2\la^2}T_0(\la+1)-\frac{q'(-\la)-q(-\la)}{2\la^2}T_0(\la-1). \label{eq:v-action-T1}
\end{align}
\item For $\la=0$ we have
\begin{equation} \label{eq:v-action-T0-singular}
(\tfrac{1}{2}v+w).T_0(0)=q(0)T_1(1)+q'(0)T_0(1).
\end{equation}
\item 
If $\la\in\C\setminus\big\{\tfrac{1}{2},\,-\tfrac{1}{2}\big\}$, then
\begin{align}
w.T_0(\la) &\in \frac{q(\la)}{1+2\la}T_0(\la+1)+\frac{q(-\la)}{1-2\la}T_0(\la-1)+ \C T_0(\la), \label{eq:w-action-T0}\\
w.T_1(\la) &\in 
\frac{q(\la)}{1+2\la}T_1(\la+1)+\frac{q(-\la)}{1-2\la}T_1(\la-1)
\nonumber\\
&+\frac{q'(\la)(\tfrac{1}{2}+\la)-q(\la)}{2(\tfrac{1}{2}+\la)^2}T_0(\la+1)+\frac{-q'(-\la)(\tfrac{1}{2}-\la)+q(-\la)}{2(\tfrac{1}{2}-\la)^2}T_0(\la-1) \nonumber \\
&+\C T_0(\la)+\C T_1(\la). \label{eq:w-action-T1}
\end{align}
\item For $\la=\tfrac{1}{2}$ we have
\begin{align}
w.T_0(\tfrac{1}{2})&\in q(-\tfrac{1}{2})T_1(\tfrac{1}{2})+\frac{1}{2}q(\tfrac{1}{2})T_0(\tfrac{3}{2})+\C T_0(\tfrac{1}{2}),
\label{eq:w-action-T0-singular}\\
w.T_1(\tfrac{1}{2})&\in \frac{1}{2}q(\tfrac{1}{2})T_1(\tfrac{3}{2})
+\frac{q'(\tfrac{1}{2})-q(\tfrac{1}{2})}{2}T_0(\tfrac{3}{2})+
\C T_0(\tfrac{1}{2})+\C T_1(\tfrac{1}{2}).
\label{eq:w-action-T1-singular}
\end{align}
\end{enumerate}
\end{subequations}
\end{Proposition}

\begin{proof}
Let $\la\in\C$ and $p(x)\in\C[x^2]$.
Since $w^\ast=-v-w$ and $(\tfrac{1}{2}v+w)^\ast=-\tfrac{1}{2}v-w$, we have
\begin{align}
\big(w.T_i(\la)\big)(p(x)) &= T_i(\la)\big(\varphi(-v-w).p(x)\big), \\
\big((\tfrac{1}{2}v+w).T_i(\la)\big)(p(x)) &= T_i(\la)\big(\varphi(-\tfrac{1}{2}v-w).p(x)\big).
\end{align}
Furthermore,
\begin{align}
\varphi(-\tfrac{1}{2}v-w).p(x)&=\frac{q(x)p(x+1)-q(-x)p(-x+1)}{2x},\\
\varphi(-v-w).p(x) &= \frac{q(x)p(x+1)-q(-\tfrac{1}{2})p(x)}{1+2x}+\frac{q(-x)p(-x+1)-q(-\tfrac{1}{2})p(-x)}{1-2x}.
\end{align}
Thus, if $\la\in\C\setminus\{0\}$ then
\begin{equation}
(\tfrac{1}{2}v+w).T_0(\la)=\frac{q(\la)}{2\la}T_0(\la+1)-\frac{q(-\la)}{2\la}T_0(\la-1).
\end{equation}
Similarly, if $\la\in\C\setminus\{\tfrac{1}{2},-\tfrac{1}{2}\}$, then
\begin{equation}
w.T_0(\la) = \frac{q(\la)}{2(\tfrac{1}{2}+\la)}T_0(\la+1)+\frac{q(-\la)}{2(\tfrac{1}{2}-\la)}T_0(\la-1)+\frac{q(-\tfrac{1}{2})}{(\tfrac{1}{2}+\la)(\tfrac{1}{2}-\la)}T_0(\la).
\end{equation}

It remains to prove the singular cases.
\begin{align*}
w.T_1(\la)&=\frac{q(\la)}{2(\tfrac{1}{2}+\la)}T_1(\la+1)+\frac{q(-\la)}{2(\tfrac{1}{2}-\la)}T_1(\la-1)+\frac{q(-\tfrac{1}{2})}{(\tfrac{1}{2}+\la)(\tfrac{1}{2}-\la)}T_1(\la) \\
&\quad+
\frac{q'(\la)(\tfrac{1}{2}+\la)-q(\la)}{2(\tfrac{1}{2}+\la)^2}T_0(\la+1)+\frac{q(-\la)}{2(\tfrac{1}{2}-\la)}T_0(\la-1)+\frac{q(-\tfrac{1}{2})}{(\tfrac{1}{2}+\la)(\tfrac{1}{2}-\la)}T_0(\la)
\end{align*}

If $\la\in \tfrac{1}{2}+\Z_{\ge 0}$ then
\begin{align*}
& T_i(\la)\Big(\frac{q(x)p(x+1)-q(-\tfrac{1}{2})p(x)}{1+2x}+\frac{q(-x)p(-x+1)-q(-\tfrac{1}{2})p(-x)}{1-2x}\Big) \\
&\quad=
\Big(\frac{q(\la)}{1+2\la}T_i(\la+1)+(-1)^i\frac{q(-\la)}{1-2\la}T_i(\la-1)-\big(\frac{q(-\tfrac{1}{2})}{1+2\la}+(-1)^i\frac{q(-\tfrac{1}{2})}{1-2\la}\big)T_i(\la)\Big)\\
&\quad+\delta_{i1}\frac{1}{2}\Big(\frac{q'(\la)(\tfrac{1}{2}+\la)-q(\la)}{(\tfrac{1}{2}+\la)^2}T_0(\la+1)+ \frac{-q'(-\la)(\tfrac{1}{2}-\la)+q(-\la)}{(\tfrac{1}{2}-\la)^2}T_0(\la-1)\Big)+ \C T_0(\la)
\end{align*}

\begin{align*}
& T_0(\tfrac{1}{2})\Big(\frac{q(x)p(x+1)-q(-\tfrac{1}{2})p(x)}{1+2x}+\frac{q(-x)p(-x+1)-q(-\tfrac{1}{2})p(-x)}{1-2x}\Big) \\
&\quad=\Big(\frac{1}{2}q(\tfrac{1}{2})T_0(\tfrac{3}{2})+\frac{q'(-\tfrac{1}{2})-q(-\tfrac{1}{2})}{2}T_0(\tfrac{1}{2})+q(-\tfrac{1}{2})T_1(\tfrac{1}{2})\Big)(p(x))
\end{align*}
Lastly, to prove \eqref{eq:w-action-T1-singular}, in the second term we Taylor expand $q(-x)$ at $x=\tfrac{1}{2}$ to get
\begin{align*}
& T_1(\tfrac{1}{2})\Big(\frac{q(x)p(x+1)-q(-\tfrac{1}{2})p(x)}{1+2x}+\frac{q(-x)p(-x+1)-q(-\tfrac{1}{2})p(-x)}{1-2x}\Big) \\
&\quad=\Big(\frac{1}{2}q(\tfrac{1}{2})T_1(\tfrac{3}{2})
+\frac{q'(\tfrac{1}{2})-q(\tfrac{1}{2})}{2}T_0(\tfrac{3}{2})+\frac{1}{2}q(-\tfrac{1}{2})T_0(\tfrac{1}{2})-q(-\tfrac{1}{2})T_1(\tfrac{1}{2})\\
&\quad-\frac{1}{4}q''(-\tfrac{1}{2})T_0(\tfrac{1}{2})-q
(-\tfrac{1}{2})T_1(\tfrac{1}{2})\Big)(p(x)).
\end{align*}
\end{proof}

We briefly describe the content of Figures \ref{fig:generic}-\ref{fig:half-integral} implied by Proposition \ref{prop:calculations}.
Let $q(t)$ be a polynomial, $\deg q(t)\ge 4$. Let $\om\in\C/\Z$. For $\om\neq\Z,\frac{1}{2}+\Z$ (the \emph{generic} case), the structure of $D(q)$-module $M_0(\om)$ is depicted in Figure \ref{fig:generic}. Each vertex represents a weight vector, i.e. an eigenvector of $u\in D(q)$. Zeros of $q(t)$ from $\om$ or $-\om$ determine submodules. For example, if $q(\la)=0$, which can be thought of as signaling that we have \emph{no} edge from $T_0(\la)$ to $T_0(\la+1)$, then $\oplus_{r=0}^\infty \C T_0(\la-r)$ is a submodule. Similarly for the \emph{singular} cases of $\om=\Z$ and $\om=\frac{1}{2}+\Z$, depicted in Figures \ref{fig:integral} and \ref{fig:half-integral} respectively. Here the vertices on top are \emph{generalized} weight vectors with respect to $u$ (``derivative tableaux''), while the remaining ones are weight vectors. The spectrum of $u$ on these modules consists of squares of integers or squares of half-integers, respectively. For these singular cases, the derivative of $q(t)$ also plays a role. Ultimately, since there can only be finitely many zeroes of $q(t)$ and $q'(t)$ in any given orbit (or at all), the submodule structure in the singular cases happens in the leftmost part of the diagrams. Eventually, far enough to the right, all edges are present and we have an infinite-dimensional irreducible submodule or quotient.

\begin{figure}
     \centering
      \begin{tikzpicture}[scale=2]
\path (0,0)   node[left] {$\cdots$}
      (0,0)   node (a) {$\bullet$}
      (a)  node[below=3mm] {$T_0(\la-1)$}
      (2,0)   node (b) {$\bullet$}
      (b)  node[below=3mm] {$T_0(\la)$}
      (4,0)   node (c) {$\bullet$}
      (c)  node[below=3mm] {$T_0(\la+1)$}
      (6,0)   node (d) {$\bullet$}
      (d)  node[below=3mm] {$T_0(\la+2)$}
      (6,0)  node[right] {$\cdots$};

\path (a) edge[->,bend left=15] node[above] {$q(\la-1)$} (b)
      (b) edge[->,bend left=15] node[below] {$q(-\la)$} (a)
      
      (b) edge[->,bend left=15] node[above] {$q(\la)$} (c)
      (c) edge[->,bend left=15] node[below] {$q(-\la-1)$} (b)
      
      (c) edge[->,bend left=15] node[above] {$q(\la+1)$} (d)
      (d) edge[->,bend left=15] node[below] {$q(-\la-2)$} (c);
\end{tikzpicture}
\caption{Structure of $M_0(\omega)$ when $\omega\neq \Z,\tfrac{1}{2}+\Z$. The label of an edge is nonzero if and only if the target is in the $(\C[u]1+\C[u]v+\C[u]w)$-span of the source.}\label{fig:generic}
\end{figure}
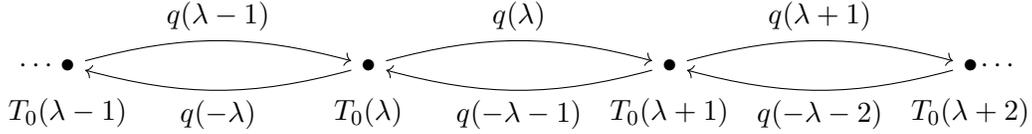
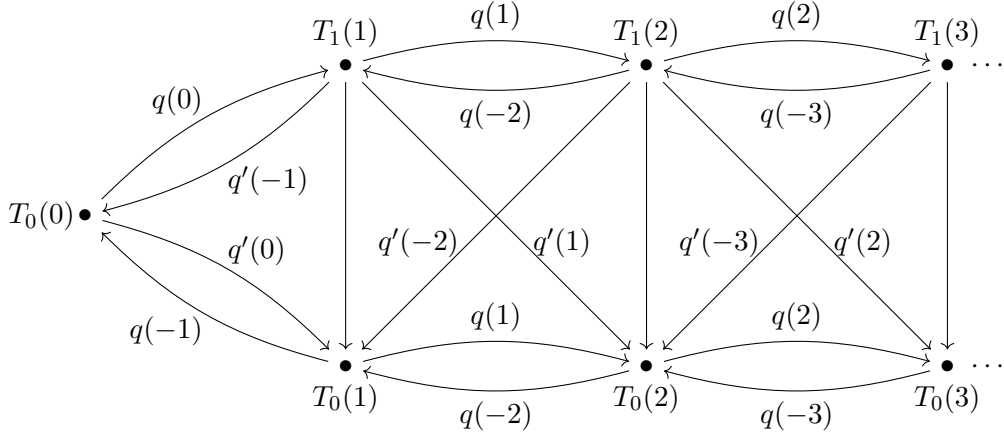
\begin{figure}
\centering
         \begin{tikzpicture}[scale=2]
\path (0,0)   node (a) {$\bullet$}
      (2,0)   node (b) {$\bullet$}
      (4,0)   node (c) {$\bullet$}
      (4,0)   node[right=5pt] {$\cdots$}
      (4,-2)  node[right=5pt] {$\cdots$}
      
      (-{sqrt(3)},-1) node (z) {$\bullet$}
      (0,-2)  node (aa) {$\bullet$}
      (2,-2)  node (bb) {$\bullet$}
      (4,-2)  node (cc) {$\bullet$}

      (z)     node[left] {$T_0(0)$}
      (aa)    node[below=3pt] {$T_0(1)$}
      (bb)    node[below=3pt] {$T_0(2)$}
      (cc)    node[below=3pt] {$T_0(3)$}
      (a)     node[above=3pt] {$T_1(1)$}
      (b)     node[above=3pt] {$T_1(2)$}
      (c)     node[above=3pt] {$T_1(3)$};
      
\path (a) edge[->,bend left=15] node[above] {$q(1)$} (b)
      (b) edge[->,bend left=15] node[below] {$q(-2)$}  (a)
      
      (b) edge[->,bend left=15] node[above] {$q(2)$}  (c)
      (c) edge[->,bend left=15] node[below] {$q(-3)$} (b)
      
      (aa) edge[->,bend left=15] node[above] {$q(1)$}  (bb)
      (bb) edge[->,bend left=15] node[below] {$q(-2)$} (aa)
      
      (bb) edge[->,bend left=15] node[above] {$q(2)$}  (cc)
      (cc) edge[->,bend left=15] node[below] {$q(-3)$} (bb)

      (z) edge[->,bend left=15] node[above left=-3pt] {$q(0)$} (a)
      (a) edge[->,bend left=15] node[below right=-3pt] {$q'(-1)$} (z)

      (z) edge[->,bend left=15] node[above right=-3pt] {$q'(0)$} (aa)
      (aa) edge[->,bend left=15] node[below left=-3pt] {$q(-1)$} (z)
      
      (a) edge[->] (aa)
      (b) edge[->] (bb)
      (c) edge[->] (cc)
      
      (a) edge[->] node[pos=0.6, right] {$q'(1)$} (bb)
      (b) edge[->] node[pos=0.6, left] {$q'(-2)$} (aa)
      
      (b) edge[->] node[pos=0.6, right] {$q'(2)$} (cc)
      (c) edge[->] node[pos=0.6, left] {$q'(-3)$} (bb);
\end{tikzpicture}
\caption{Structure of the $D(q)$-module $M(\Z)$. The label of an edge is nonzero if and only if the target is in the $(\C[u]1+\C[u]v+\C[u]w)$-span of the source, and $q'(x)$ means the derivative of $q(x)$.}\label{fig:integral}
\end{figure}
\begin{figure}
\centering
\begin{tikzpicture}[scale=2]
\path (0,0)   node (a) {$\bullet$}
      (2,0)   node (b) {$\bullet$}
      (4,0)   node (c) {$\bullet$}
      (4,0)   node[right=5pt] {$\cdots$}
      (4,-2)  node[right=5pt] {$\cdots$}
      
      (0,-2)  node (aa) {$\bullet$}
      (2,-2)  node (bb) {$\bullet$}
      (4,-2)  node (cc) {$\bullet$}

      (aa)    node[below=3pt] { $T_0(\tfrac{1}{2})$}
      (bb)    node[below=3pt] { $T_0(\tfrac{3}{2})$}
      (cc)    node[below=3pt] { $T_0(\tfrac{5}{2})$}
      (a)     node[above=3pt] { $T_1(\tfrac{1}{2})$}
      (b)     node[above=3pt] { $T_1(\tfrac{3}{2})$}
      (c)     node[above=3pt] { $T_1(\tfrac{5}{2})$};
      
\path (a) edge[->,bend left=15] node[above] {$q(\tfrac{1}{2})$} (b)
      (b) edge[->,bend left=15] node[below] {$q(-\tfrac{3}{2})$}  (a)
      
      (b) edge[->,bend left=15] node[above] {$q(\tfrac{3}{2})$}  (c)
      (c) edge[->,bend left=15] node[below] {$q(-\tfrac{5}{2})$} (b)
      
      (aa) edge[->,bend left=15] node[above] {$q(\tfrac{1}{2})$}  (bb)
      (bb) edge[->,bend left=15] node[below] {$q(-\tfrac{3}{2})$} (aa)
      
      (bb) edge[->,bend left=15] node[above] {$q(\tfrac{3}{2})$}  (cc)
      (cc) edge[->,bend left=15] node[below] {$q(-\tfrac{5}{2})$} (bb)

      (aa) edge[->,bend left=90] node[below left=-3pt] {$q(-\tfrac{1}{2})$} (a)
      
      (a) edge[->] (aa)
      (b) edge[->] (bb)
      (c) edge[->] (cc)
      
      (a) edge[->] node[pos=0.6, right] {$q'(\tfrac{1}{2})$} (bb)
      (b) edge[->] node[pos=0.6, left] {$q'(-\tfrac{3}{2})$} (aa)
      
      (b) edge[->] node[pos=0.6, right] {$q'(\tfrac{3}{2})$} (cc)
      (c) edge[->] node[pos=0.6, left] {$q'(-\tfrac{5}{2})$} (bb);
\end{tikzpicture}
\caption{Structure of the $D(q)$-module $M(\tfrac{1}{2}+\Z)$. The label of an edge is nonzero if and only if the target is in the $(\C[u]1+\C[u]v+\C[u]w)$-span of the source, and $q'(x)$ means the derivative of $q(x)$.}\label{fig:half-integral}
\end{figure}

\bibliographystyle{siam}

\end{document}